\documentclass[11pt]{amsart}

\usepackage[T1]{fontenc}
\usepackage{fourier}
\usepackage{amsmath,amssymb}
\usepackage{graphicx}
\usepackage{import}
\usepackage{amscd}
\usepackage{wrapfig}
\usepackage{fullpage}
\usepackage{epsfig}
\numberwithin{equation}{section}

\newtheorem{theorem}{Theorem}[section]

\newtheorem{lemma}[theorem]{Lemma}
\newtheorem{corollary}[theorem]{Corollary}
\newtheorem{proposition}[theorem]{Proposition}
\newtheorem{claim}[theorem]{Claim}

\theoremstyle{definition}
\newtheorem{definition}[theorem]{Definition}
\newtheorem{notation}[theorem]{Notation}

\theoremstyle{remark}
\newtheorem{remark}[theorem]{Remark}

\newcommand{\G}{{\mathcal G}}
\newcommand{\N}{{\mathbb N}}
\newcommand{\R}{{\mathbb R}}
\newcommand{\Z}{{\mathbb Z}}

\newcommand{\x}{{\mathbf{x}}}
\newcommand{\y}{{\mathbf{y}}}

\newcommand{\J}{{\mathcal J}}
\newcommand{\A}{{\mathcal A}}

\newcommand{\bP}{\mathbb{P}}
\newcommand{\cross}{\times}

\newcommand{\mcg}{\mathcal{MC\kern0.04emG}}
\newcommand{\ml}{\mathcal{M \kern0.07emL}}
\newcommand{\pmf}{\mathcal{P \kern0.07em M \kern0.07em F}}

\newcommand{\cG}{\mathcal{G}}

\newcommand{\cR}{\mathcal{R}}

\newcommand{\teichmuller}{Teichm{\"u}ller{ }}

\makeatletter
 \let\c@theorem=\c@subsection
 \let\c@conjecture=\c@subsection
 \let\c@lemma=\c@subsection
 \let\c@proposition=\c@subsection
 \let\c@claim=\c@subsection
 \let\c@question=\c@subsection
 \let\c@criterion=\c@subsection
 \let\c@vfconj=\c@subsection
 \let\c@definition=\c@subsection
 \let\c@notation=\c@subsection
 \let\c@remark=\c@subsection
 \let\c@example=\c@subsection
 \let\c@equation=\c@subsection
 \let\c@figure=\c@subsection
 \let\c@wrapfigure=\c@subsection

\makeatother

\begin{document}

\title[linear inv]{Every transformation is disjoint from almost every non-classical exchange.}
\author[Chaika]{Jon Chaika}
\address{Mathematics 209, University of Utah, 155 S 1400 E RM 233, Salt Lake City, UT 84112, USA.}
\email{chaika@math.utah.edu}
\author[Gadre]{Vaibhav Gadre}
\address{IAS and Mathematics Institute, University of Warwick, Coventry, CV4 7AL, UK.}
\email{gadre.vaibhav@gmail.com}

\begin{abstract}
A natural generalization of interval exchange maps are linear involutions, first introduced by Danthony and Nogueira \cite{Dan-Nog}. Recurrent train tracks with a single switch which we call non-classical interval exchanges \cite{Gad}, form a subclass of linear involutions without flips. They are analogs of classical interval exchanges, and are first return maps for non-orientable measured foliations associated to quadratic differentials on Riemann surfaces. We show that every transformation is disjoint from almost every irreducible non-classical interval exchange. In the appendix,  we prove that for almost every pair of quadratic differentials with respect to the Masur-Veech measure, the vertical flows are disjoint. 
\end{abstract}

\maketitle

%%%%%%%%%%%%%%%%%%%%%%%%%%%%%%%%%%%%
\section{Introduction}
%%%%%%%%%%%%%%%%%%%%%%%%%%%%%%%%%%%%
Here, we are interested in a dynamical property called disjointness for irreducible non-classical exchanges, a subclass of linear involutions without flips \cite{Dan-Nog}. Classical interval exchange maps arise as first return maps to intervals transverse to vertical foliations of abelian differentials on Riemann surfaces. Analogously, non-classical exchanges are first return maps to pairs of disjoint parallel transverse intervals to non-orientable foliations associated to quadratic differentials. Briefly, non-classical exchange are invertible piecewise isometries of a pair of disjoint interval decomposed into finitely many pieces, and satisfying some additional assumptions (\cite[Definition 2.1]{Boi-Lan}). Equivalently, non-classical exchanges can be pictorially defined as recurrent train tracks with a single switch on a Riemann surface. The definitions are presented in detail in Section \ref{niem}.

Two measure preserving transformations $(T,X, \mu)$ and $(S, Y, \nu)$ are said to be {\em disjoint} if the product measure $\mu \times \nu$ is the only invariant measure for the product transformation $T \times S$ with marginals $\mu$ and $\nu$. It is a way of saying that two measure preserving transformations are different and in particular, implies that they are not isomorphic. It was shown in \cite{Cha} that every ergodic transformation is disjoint from almost every classical exchange. Here, we prove the analogous result for non-classical exchanges under a mild technical condition. 

\begin{theorem}\label{disjoint}
Let $T: X \to X$ be $\mu$-ergodic. Then $T$ is disjoint with respect to almost every irreducible non-classical exchange with an orientation preserving band.
\end{theorem}

Almost every irreducible non-classical interval exchange is uniquely ergodic \cite{Mas}. See also \cite{Gad} for a proof with techniques similar to this paper. As a result of this we obtain the following corollary:
\begin{corollary} For any uniquely ergodic non-classical exchange $S_1$ and almost every irreducible non-classical exchange $S_2$ with an orientation preserving band, the product transformation $S_1 \times S_2$ is uniquely ergodic.
\end{corollary}
\begin{remark}
The results above are false if there are no orientation preserving bands. As will be clear from the definition, if a non-classical exchange $\x$ has no orientation preserving bands then $\x^2$ leaves $I_+$ and $I_-$ invariant, so it is not even minimal.
\end{remark}

In classical interval exchanges, an interval $I$ is partitioned into $d$ subintervals, these subintervals are permuted and re-glued preserving orientation to get an invertible, piecewise order preserving, piecewise isometry of $I$. The widths of the subintervals and the permutation used for gluing completely determine a classical interval exchange. There is a way to draw these maps pictorially:

\begin{wrapfigure}{r}{0.33\textwidth}
\begin{center}
\ \epsfig{file=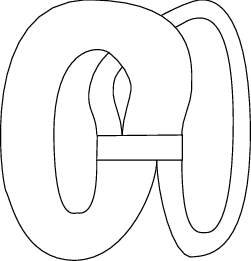, width=0.25\textwidth, height= 0.18\textheight}
\end{center}
\caption{A classical interval exchange.}
\label{class-iem}
\end{wrapfigure}

We draw the original interval $I$ horizontally and then thicken it vertically to get two copies $I_+$ and $I_-$, the top and the bottom intervals. Let $\epsilon: I_+ \sqcup I_- \to I_+\sqcup I_-$ be the map that switches the intervals i.e., $\epsilon(I_+) = I_-$ and $\epsilon(I_-) = I_+$. Divide $I_+$ into $d$ subintervals with the prescribed widths. Divide $I_-$ also into $d$ subintervals but incorporate the permutation to decide the widths. Thus, each subinterval of $I_+$  pairs off by a translation with a subinterval of $I_-$ with the same width whose placement is determined by the permutation. We join the pair of subintervals by a band of uniform width. As an example, Figure~\ref{class-iem} shows a classical exchange with 2 bands, which is a rotation. 

The composition of the map from $I_+$ to $I_-$ given by the vertical flow up from $I_+$ along the bands, followed by $\epsilon$, exhibits the classical exchange, as a map from $I_+$ to $I_+$. The inverse of this exchange is then the map from $I_-$ to $I_-$ given by flowing the other way. Thus, we have an equivalent formulation of the interval exchange as a Lebesgue measure preserving transformation on the disjoint union $I_+ \sqcup I_-$. We shall use this formulation since non-classical exchanges are defined in a similar way: the intervals $I_+$ and $I_-$ are partitioned into subintervals such that paired off intervals are joined by bands of uniform width, but now we allow bands to go from $I_+$ to $I_+$ and also from $I_-$ to $I_-$. The non-classical exchange as a map from $I_+ \sqcup I_-$ to itself is defined by flow along the bands followed by $\epsilon$. We shall call the bands that go from  $I_+$ to $I_+$ and from $I_-$ to $I_-$ {\em orientation reversing}, since the exchange restricted to the subintervals given by the ends of such a band is orientation reversing. 

A classical exchange is said to be {\em irreducible}, if the permutation $\pi$ used in the gluing is {\em irreducible}: for all subsets $\{ 1, \cdots , k\}, k < d$, we have $\pi( \{ 1, \cdots, k\} ) \neq \{1, \cdots, k\}$. Pictorially, a reducible classical exchange is a concatenation of two classical exchanges side by side for all widths of the bands, and the dynamics can be studied by restricting to each piece. For technical reasons, the appropriate notion of irreducibility for non-classical exchanges is subtler than the straightforward irreducibility notion above. This was done by Boissy and Lanneau in \cite{Boi-Lan}. Here, we consider only those non-classical exchanges that satisfy the Boissy and Lanneau definition of irreducibility. The precise definitions and details will be provided in Section~\ref{niem}.

For a fixed permutation of a classical exchange, the set of widths normalized to have sum 1, is the standard simplex in $\R^d$, and carries a natural Lebesgue measure. Thus, the full parameter space is a disjoint union of simplices over irreducible permutations of $d$ letters. For non-classical exchanges, after fixing the combinatorics, the set of normalized widths satisfies an additional switch condition that the sum of the widths of orientation reversing bands of $I_+$ is equal to the sum of widths of orientation reversing bands of $I_-$. Thus, it has codimension 1 in the standard simplex, and  inherits a natural Lebesgue measure. We shall call it the {\em configuration space} associated to the combinatorics. The full parameter space is the disjoint union of the configuration spaces. 

Following \cite{Cha}, we use a criteria of Hahn and Parry that two ergodic transformations on a pair of measures spaces are disjoint if the induced maps on $L^2$-functions on the spaces are spectrally singular. Our proof follows the proof in \cite{Cha} while verifying the individual steps for non-classical exchanges. 

A key step in the proof in \cite{Cha} are rigidity sequences for classical exchanges. The existence of these sequences follows from a cyclic approximation theorem of Veech \cite{Vee2}. We prove a similar cyclic approximation theorem for non-classical exchanges, and the existence of rigidity sequences follows from it. 

As an analog of Theorem \ref{disjoint}, we prove disjointness of vertical flows for quadratic differentials in the appendix. To be precise, we show:

\begin{theorem}\label{flow disjoint} With respect to the Masur-Veech measure, for almost every pair of quadratic differentials the vertical flows are disjoint.
\end{theorem}

The proof of Theorem \ref{flow disjoint} is simpler and relies on mixing of the \teichmuller geodesic flow.

\subsection{Outline of the paper:} In Section 2, we define non-classical exchanges and explain irreducibility for them. In Section 3, we define Rauzy induction for non-classical exchanges, and state the key result of Boissy and Lanneau \cite{Boi-Lan} for attractors of the Rauzy diagram. In Section 4, we summarize the results from \cite{Gad} for expansions of non-classical exchanges by iterated Rauzy induction. In Section 5, we state and prove the cyclic approximation theorem for non-classical exchanges. In Section 6, we establish the existence of rigidity sequences. In Section 7, we prove total ergodicity for almost every non-classical exchange with an orientation preserving band present. In Section 8, we assemble the ingredients to conclude the proof of the main theorem verbatim from \cite{Cha}. In the appendix, we prove disjointness of vertical flows for quadratic differentials.

\subsection{Acknowledgements:} Chaika was supported in part by NSF grant DMS-1004372. Gadre was supported by a Simons Travel Grant. The authors thank the 2011 Park City Math Institute program.

%%%%%%%%%%%%%%%%%%%%%%%%%%%%%%%%%%%%
\section{Non-classical Interval Exchanges}\label{niem}
%%%%%%%%%%%%%%%%%%%%%%%%%%%%%%%%%%%%

Let $\A$ denote an alphabet over $d$ letters. In the definition that follows, the set $\A$ labels the bands. A classical exchange is determined by the widths $(\lambda_\alpha), \alpha \in \A$ of the subintervals and bijections $p_0$ and $p_1$ from $\A$ to the set $\{1, \dotsc, d\}$ as follows: In the plane, draw two vertically aligned horizontal copies $I_+$ and $I_-$ of the interval $I = [0,\sum_\alpha \lambda_\alpha)$. Call them the top interval and the bottom interval respectively. Let $\epsilon: I_+ \sqcup I_- \to I_+ \sqcup I_-$ be the map that switches the intervals i.e., $\epsilon(I_+) = I_-$ and $\epsilon(I_-) = I_+$. Subdivide $I_+$ into $d$ subintervals with widths $\lambda_{p_0^{-1}(1)}, \dotsc, \lambda_{p_0^{-1}(d)}$ from left to right. Subdivide $I_-$ into $d$ subintervals with widths $\lambda_{p_1^{-1}(1)}, \dotsc, \lambda_{p_1^{-1}(d)}$ from left to right.  For each $\alpha \in \A$, join the $p_0(\alpha)$ subinterval of $I_+$ to the $p_1(\alpha)$ subinterval of $I_-$ by a band of uniform width $\lambda_\alpha$.  The vertical flow along the bands from $I_+$ to $I_-$, followed by $\epsilon$ exhibits the classical exchange as a map from $I_+$ to $I_+$. Similarly, the inverse of the interval exchange is realized as a map from $I_-$ to itself by flowing reverse along the bands, followed by $\epsilon$. For classical exchanges, every band has one end on $I_+$ and the other on $I_-$.

To define non-classical exchanges, the intervals $I_+$ and $I_-$ are partitioned into subintervals that come in pairs, so that they can be joined by bands with uniform width. In this case however, we require bands that have both ends in $I_+$, and bands that have both ends in $I_-$. Such a band shall be called {\em orientation reversing} because the non-classical exchange restricted to the subintervals given by the ends of the band is orientation reversing. As before, the exchange as a map $T: I_+ \sqcup I_- \to I_+ \sqcup I_-$ is given by flowing along the band away from the subinterval in question, followed by $\epsilon$. The Lebesgue measure $l$ on $I_+ \sqcup I_-$ is obviously invariant under $T$. 

\begin{wrapfigure}{r}{0.36\textwidth}
\begin{center}
\ \epsfig{file=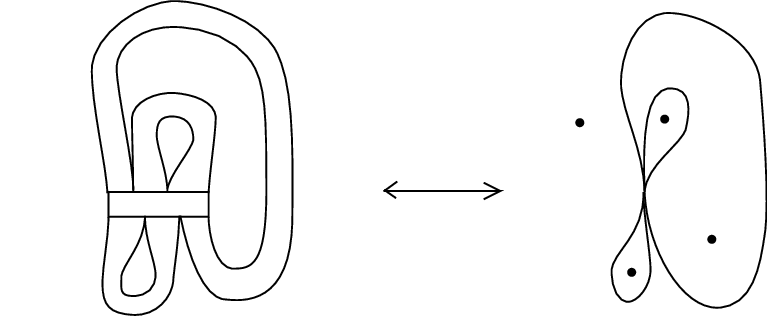, width=0.31\textwidth}
\end{center}
\vskip 5pt \caption{A non-classical exchange on a 4-puntured sphere.}
\label{iem}
\end{wrapfigure}

We shall work with {\em labelled} non-classical exchanges i.e.,~there is a bijection from $\A$ to the set of bands. As defined by \cite{Boi-Lan}, the labeling can be thought of as given by {\em a generalized permutation} which is a 2-1 map from the set $\{1, \dotsc , 2d\}$ to $\A$. Thus, $\pi^{-1} \alpha$ denote the ends of the band $\alpha$. A generalized permutation $\pi$ is of type $(l,m)$ where $l+m=2d$ if the set $\{1, \dotsc, l\}$ enumerates the subintervals of $I_+$ from left to right and the set $\{l+1, \dotsc, l+m=2d\}$ enumerates the subintervals of $I_-$ from left to right. A generalized permutation defines a fixed point free involution $\sigma$ of $\{1, \dotsc, 2d\}$ by:
\[
\pi^{-1}(\pi(i)) = \{ i , \sigma(i)\}.
\]
The equivalence classes under $\sigma$ can be indexed with the elements of $\A$ and correspond to the bands in our picture. The generalized permutations encoding a non-classical exchange do not arise from a true permutation $p= p_1p_0^{-1}$ i.e., it has a positive integer $i$ with $i,  \sigma(i) \leqslant l$ and a positive integer $j$ with $l+1 \leqslant j, \sigma(j)$. This means that there are orientation reversing bands for $I_+$ and $I_-$. Following Kerckhoff, we shall call the positions that are rightmost on the intervals $I_+$ and $I_-$, the {\em critical} positions. We let $I(\alpha)$ be the union of subintervals at the ends of band $\alpha$. 

After a horizontal isotopy collapsing $I_\pm$ to points, the spine of the picture can be thought of as an abstract train track $\tau$. In particular, if $\tau$ admits an embedding into an orientable surface $S$ as a {\em large} train track i.e., with complementary regions that are polygons or once-punctured polygons, then the non-classical exchanges associated to the combinatorics, give measured foliations on $S$ belonging to the stratum of quadratic differentials indicated by the complementary regions. For example, Figure~\ref{iem} shows a non-classical exchange on a 4-punctured sphere, in the principle stratum.

\subsection{Irreducibility:} 

A classical exchange is called {\em irreducible} if there is no $k< d$ such that $p_0^{-1} \{ 1, \cdots, k \} = p_1^{-1} \{ 1,\cdots, k \}$. In other words, a reducible classical exchange is a concatenation of two classical exchanges over the subsets $\{1, \cdots, k\}$ and $\{ k+1, \cdots, d\}$. In terms of our picture, for all widths of the subintervals, the intervals $I_\pm$ partition into two subintervals $I_\pm(1)$ and $I_\pm(2)$, and the set of bands partition into two sets $\A_1$ and $\A_2$ such that all bands in $\A_1$ have both ends in $I_\pm(1)$, and all bands in $\A_2$ have both ends in $I_\pm(2)$. 

For non-classical exchanges, a straightforward combinatorial reduction as above can be considered. However, for technical reasons described in \cite{Boi-Lan}, a broader notion of reducibility becomes necessary. From now on, we shall assume that the non-classical exchanges considered are irreducible in this sense of Boissy and Lanneau. See the end of Section 3 for more details.

%%%%%%%%%%%%%%%%%%%%%%%%%%%%%%%%%%%%
\section{Rauzy induction}\label{Rauzy}
%%%%%%%%%%%%%%%%%%%%%%%%%%%%%%%%%%%%

A key technical tool is Rauzy induction, which is simply the first return map to an appropriate pair of subintervals $I'_\pm$. A precise definition is also given in Section 2.2 of \cite{Boi-Lan}. Here, we describe it in terms of the picture, and focus on the coding of iterations by products of elementary matrices. These matrices describe the induced map on the parameter space. 

Iterations of Rauzy induction are analogous to continued fraction expansions. In fact, when a classical exchange has two bands, the expansion is equivalent to the continued fraction expansion of the ratio of their widths.

\begin{figure}[htb]
\begin{center}
\ \psfig{file=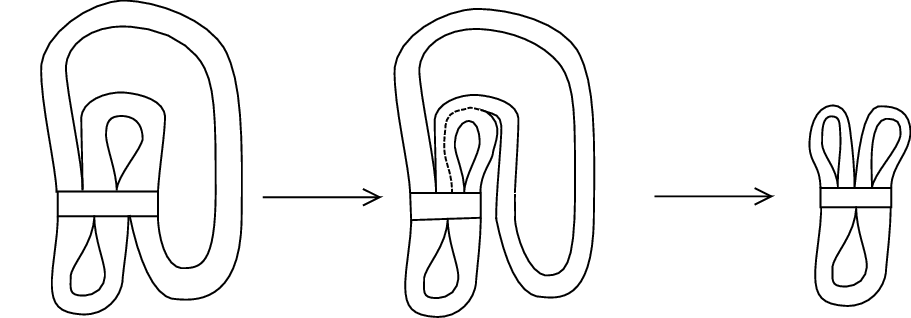, height=1.5truein, width=5truein} \caption{Rauzy Induction} \label{rauzyind}
\end{center}
\setlength{\unitlength}{1in}
\begin{picture}(0,0)(0,0)
\put(1.3,1.6){$\alpha_0$} \put(1.6,0.4){$\alpha_1$} \put(5.1,1.5){$\alpha_1$} \put(5.5,1.5){$\alpha_0$} \put(5.3,0.4){$\alpha'_1$} \put(0.3,1.1){$\lambda_{\alpha_0} = 3/7$} \put(0.3,0.9){$\lambda_{\alpha_1} = 1/7$} \put(5.7,1.1){$\lambda'_{\alpha_0} = 2/7$} \put(5.7,0.9){$\lambda'_{\alpha_1}= 1/7$}
\end{picture}
\end{figure}
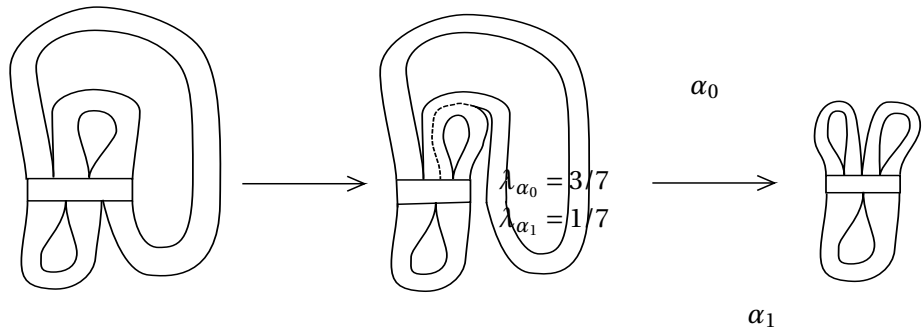

Let $T$ be a non-classical exchange. Let $\alpha_0$ and $\alpha_1$ be the bands in the critical positions with $\alpha_0$ on $I_+$. First, suppose that $\lambda_{\alpha_0} > \lambda_{\alpha_1}$. Then we slice as shown in Figure~\ref{rauzyind} till we hit $I_+ \sqcup I_-$ for the first time. The $\alpha_0$ band remains in its critical position, but typically a different band $\alpha'_1$ moves into the other critical position. Furthermore, the new width of $\alpha_0$ is $ \lambda_{\alpha_0} - \lambda_{\alpha_1}$. All other widths remain unchanged. If instead $\lambda_{\alpha_1} < \lambda_{\alpha_0}$, then we slice in the opposite direction. In either case, we get a new non-classical exchange with combinatorics and widths as described above. This operation is called {\em Rauzy} induction. Since Rauzy induction is represented pictorially by one band splitting another, we shall simply call it a {\em split}. In fact, this is consistent with the notion of a split in the context of the underlying train-tracks. Iterations of Rauzy induction are therefore called {\em splitting sequences}.

In the first instance above, let $I' = [0,\sum_{\alpha \neq \alpha_1} \lambda_{\alpha})$, and in the second, let $I' = [0, \sum_{\alpha \neq \alpha_0} \lambda_{\alpha})$. Let $I_+(1)$ and $I_-(1)$ denote the copies of $I'$ in $I_+$ and $I_-$ respectively. Rauzy induction is then the first return map to $I_+(1) \sqcup I_-(1)$. Let $\cR(T)$ denote the non-classical exchange induced on $I_+(1) \sqcup I_-(1)$. Iteratively, let $I_\pm(n)$ be the (nested) subintervals of $I_\pm$ to which $n$ splits give the first return map. Let $\cR^n(T)$ be the non-classical exchange induced on $I_+(n) \sqcup I_-(n)$. 

 In general, not all instances of Rauzy induction are defined. To enumerate:
\begin{enumerate}
\item When $\alpha_0 = \alpha_1$ neither of the splits are defined. 
\item When $\lambda_{\alpha_0} = \lambda_{\alpha_1}$ then neither of the splits is defined.
\item When $\alpha_0$ is an orientation reversing band on $I_+$ and $\alpha_1$ is the only orientation reversing band on $I_-$, then $\alpha_0$ can split $\alpha_1$ but not the other way round i.e., only one of the splits is defined.
\end{enumerate}
Case (1) is ruled out when the non-classical exchange is irreducible and Case (2) represents a set of measure zero. In fact, as shown in Section of \cite{Gad}, almost surely, iterations of Rauzy induction continue ad infinitum. 

 \subsection{Encoding expansions by matrices:}

\subsubsection{Description of the parameter space:}
Consider the vector space $\R^{\A}$ and let $\R_{\geqslant 0}^{\A}$ be the set of points with non-negative coordinates. Let $\Delta$ denote the standard $(d-1)$-simplex in $\R^{\A}$ with sum of the coordinates equal to 1. An assignment of widths to the bands is a point in $\R^{\A}$. Normalizing the widths so that their sum is 1 restricts us to $\Delta$.  

To be consistent with a generalized permutation $\pi$, any assignment of widths must satisfy the switch condition defined by $\pi$. We denote the set of such widths normalized to sum 1 by $W(\pi)$, and un-normalized widths by $\R^+ W(\pi)$. Let $\A_+$ and $\A_-$ be the set of orientation reversing bands incident on $I_+$ and $I_-$ respectively. The points in $W(\pi)$ satisfy the additional constraint: 
\[
\sum_{\alpha \in \A_+} \lambda_{\alpha}  = \sum_{\alpha \in \A_-} \lambda_{\alpha}.
\]
Thus $W(\pi)$ is the intersection with $\Delta$ of a codimension 1 subspace of $\R^\A$.  For $\alpha \in \A_+$ and $\beta \in \A_-$, let $e_{\alpha \beta}$ be the midpoint of the edge $[e_\alpha, e_\beta]$ of $\Delta$ joining the vertices $e_{\alpha}$ and $e_{\beta}$. The subset $W(\pi)$ is the convex hull of the points $e_{\alpha \beta}$ and $e_{\rho}$ for $\rho \notin \A_+ \cup \A_-$.

There are finitely many generalized permutations $\pi$ of an alphabet $\A$ over $d$ letters, and hence finitely many convex codimension 1 subsets of $\Delta$ that could be $W(\pi)$. We call the $W(\pi)$ {\em configuration spaces}. The full parameter space is a disjoint union of the configuration spaces $W(\pi)$. 

\subsubsection{Matrices:}
Let $I$ denote the $d \times d$ identity matrix on $\R^\A$. For $\alpha, \beta \in \A$, let $M_{\alpha \beta}$ be the $d \times d$-matrix with the $(\alpha,\beta)$ entry 1 and all other entries 0. After Rauzy induction, the relationship between the old and new width data is expressed by
\[
\lambda = E \lambda',
\]
where the matrix $E$ has the form $E = I + M$. In the first instance of the split, when $\lambda_{\alpha_0} > \lambda_{\alpha_1}$, the matrix $M = M_{\alpha_0\alpha_1}$; in the second instance of the split, when $\lambda_{\alpha_1} > \lambda_{\alpha_0}$, the matrix $M = M_{\alpha_1 \alpha_0}$. Thus, in either case the matrix $E$ is an elementary matrix, in particular $E \in SL(d; \Z)$. If $B$ is any $d \cross d$ matrix then in the instance when $\lambda_{\alpha_0} > \lambda_{\alpha_1}$, the action on $B$ by right multiplication by $E$ has the effect that the
$\alpha_1$-th column of $B$ is replaced by the sum of the $\alpha_0$-th column and $\alpha_1$-th column of $B$. We phrase this as: in the split, the $\alpha_0$-th column {\em adds to} the $\alpha_1$-th column. A similar statement holds when $\lambda_{\alpha_0} > \lambda_{\alpha_1}$.

\subsection{Rauzy diagram}\label{Rauzy-diagram}

For non-classical exchanges, one constructs an oriented graph similar to the Rauzy diagram for a classical exchange. However, there are some key differences in this context.

Construct an oriented graph $G$ as follows: the nodes of the graph are generalized permutations $\pi$ of an alphabet $\A$ over $d$ letters satisfying the conditions in the third paragraph at the beginning of Section 3. We draw an arrow from $\pi$ to $\pi'$, if $\pi'$ results from splitting $\pi$. For each node $\pi$, there are at most two arrows coming out of it. A splitting sequence gives us a directed path in $G$.

For irreducible classical exchanges, each connected component of the Rauzy diagram is an {\em attractor} i.e., any node can be joined to any other node by a directed path. Each component is called a Rauzy class. 

The Rauzy diagram for non-classical exchanges is more complicated and need not have such strong recurrence properties. See the examples in the Appendix of \cite{Boi-Lan} or see Section 10 of \cite{Dun-DTh}. In particular, the straightforward definition of reducibility does not work from the point of view of the Rauzy diagram: there are generalized permutations that are not obviously reducible that can split to an obviously reducible one with positive probability. Hence, the broader definition of reducibility defined by Boissy and Lanneau \cite{Boi-Lan}, becomes necessary. In \cite{Boi-Lan}, they prove that

\begin{theorem}[\cite{Boi-Lan} Theorem C ]\label{attractors}
Let $G_{irr}$ be the subset of nodes of $G$ corresponding to the strongly irreducible generalized permutations. Then $G_{irr}$ is closed under forward iterations of Rauzy induction. Moreover, each connected component of $G_{irr}$ is strongly connected i.e., any node $\pi$ in a connected component of $G_{irr}$ can be connected to any other node $\pi'$ in the same component of $G_{irr}$ by a sequence of splits.
\end{theorem}

For the rest of the paper, the generalized permutations considered will be irreducible in the stronger sense, and in a single attractor $\cG$ of $G$. 

%%%%%%%%%%%%%%%%%%%%%%%%%%%%%%%%%%%%
\section{Dynamics}\label{dynamics}
%%%%%%%%%%%%%%%%%%%%%%%%%%%%%%%%%%%%

In this section, we analyze the expansion by splitting sequences on the parameter space. 

\subsubsection{Preliminary notation:}
Given a matrix $A$ with non-negative entries, we define the projectivization $TA$ as a map from $\Delta$ to itself by
\[
\bP A(\y) = \frac{A \y}{\vert A \y \vert}.
\]
where if $\y = (y_1, y_2, \cdots, y_d)$ in coordinates then $\vert \y \vert = \sum \vert y_i \vert$. This shall be the norm used throughout. The norm is additive on $\R_{\geqslant 0}^{d}$, i.e.~for $\y, \y'$ in $\R_{\geqslant 0}^{d}$, $\vert \y + \y'\vert = \vert \y \vert + \vert \y' \vert$.

\subsection{Iterations of Rauzy induction:}
Let $\pi_0  \in \cG$. The non-classical exchanges with generalized permutation $\pi_0$ are points in the configuration space $W(\pi_0)$. Let $\x = (\lambda_\alpha) \in W(\pi_0)$. As shown in \cite{Gad}, almost every $\x $ has an infinite expansion. An infinite expansion determines an infinite directed path in $\cG$. 

A finite directed path $\jmath: \pi_0  \to \pi_1 \to \dotsc \to \pi_n$ in $\cG$ shall be called a {\em stage} in the expansion. Let $E_i$ denote the elementary matrix associated to the split $\pi_{i-1} \to \pi_{i}$. The image of $W(\pi_i)$ under the projective linear map $\bP E_i$ lies in $W(\pi_{i-1})$. The matrix $Q_\jmath$ associated to the stage is given by 
\[
Q_\jmath = E_1 E_2 \dotsc E_n.
\]
The set $\bP Q_\jmath(W(\pi_n))$ is the set of all $\x \in W(\pi_0)$ whose expansion begins with $\pi_0 \to \pi_1 \to \cdots \to \pi_n$. 

In the expansion of $\x$, whenever it is necessary to emphasize the dependence on $\x$, we shall denote the nodes by $\pi_{\x,i}$, the configuration spaces defined by $\pi_{\x, i}$ by $W(\pi_{\x,i})$, and the elementary matrices associated to $\pi_{\x,i-1} \to \pi_{\x,i}$ by $E_{x,i}$. Thus, given a stage $\jmath: \pi_0  \to \pi_1 \to \dotsc \to \pi_n$, the set $\bP Q_\jmath(W(\pi_n))$ is precisely the set of all $\x \in W(\pi_0)$ for which $\pi_{\x,i} = \pi_i$ for all $i \leqslant n$. 

The actual (or un-normalized) widths $\lambda^{(n)}$ are given by 
\[
\x = Q_\jmath \lambda^{(n)}.
\]
The projectivization $\x^{(n)} = \lambda^{(n)}/ \vert \lambda^{(n)} \vert$ belongs to $W(\pi_n)$. Recall that $\cR^n(\x)$ is the non-classical exchange on subintervals $I_\pm(n)$ induced by $\x$. The widths of the bands in $\cR^n(\x)$ are exactly $\lambda^{(n)}$.  Let $I(\alpha, n)$ denote the union of subintervals of $I_\pm(n)$ given by the ends of band $\alpha$. 

For a point $t \in I_+(n) \sqcup I_-(n)$, let $m_n(t)$ be the first return time to $I_+(n) \sqcup I_-(n)$ under the transformation $\x$. It follows immediately that $m_n(t)$ is constant on each $I(\alpha, n)$. So for any $t \in I(\alpha, n)$, we will denote the return time by $m_n(\alpha)$. For $t \in I(\alpha, n)$, consider the finite set 
\[
S_\alpha(n) = \{t, \x(t), \x^2(t), \cdots, \x^{m_n-1}(t) \}.
\]
\begin{lemma}\label{matrix-elements}
The $(\beta, \alpha)$ entry of $Q_n$ counts the number of points in $S_\alpha(n)$ that lie in $I(\beta)$. In other words, the $(\beta, \alpha)$ entry of $Q_n$ counts the number of times a point in $I(\alpha, n)$ visits $I(\beta)$ under the exchange $\x$, before the first return to $I_+(n) \sqcup I_-(n)$. 
\end{lemma}
\begin{proof}
The proof proceeds by induction. For $n=1$, if $\alpha$ does not split some other band in $\x \to \cR(\x)$, then $m_1(\alpha) = 1$. In the same situation, all $(\beta, \alpha)$ entries of $Q_1$ are zero and the $(\alpha, \alpha)$ entry is 1, verifying the lemma. If $\alpha$ splits $\beta$ in $\x \to \cR(\x)$, then $m_1(\alpha)  = 2$ on $I(\alpha, 1)$ and in fact $\#S_\alpha(1) \cap I(\beta) = 1$ which is the same as the $(\beta, \alpha)$ entry of $Q_1$, again verifying the lemma.

Now suppose that the lemma is true for $n-1$. If $\alpha$ does not split some other band in $\cR^{n-1}(\x) \to \cR^n(\x)$ then $m_n(\alpha) = m_{n-1}(\alpha)$ and in fact $\#S_\alpha(n) \cap I(\beta) = \#S_\alpha(n-1) \cap I(\beta)$ for all $\beta$. By induction $\#S_\alpha(n-1) \cap I(\beta)$ is the same as the $(\beta, \alpha)$ entry of $Q_{n-1}$, which is the same as the $(\beta, \alpha)$ entry of $Q_{n}$, verifying the lemma in this case. 

On the other hand, if $\alpha$ splits $\beta$ in $\cR^{n-1}(\x) \to \cR^n(\x)$, then $m_n(\alpha) = m_{n-1}(\alpha) + m_{n-1}(\beta)$ and in fact $\#S_\alpha(n) \cap I(\gamma) = \#S_\alpha(n-1) \cap I(\gamma) + \#S_\beta(n-1) \cap I(\gamma)$. By induction, the right hand side is the sum of the $(\gamma, \alpha)$ and $(\gamma, \beta)$ entry of $Q_{n-1}$ which is the same as $(\gamma, \alpha)$ entry of $Q_n$. So the lemma is verified in this case too.
\end{proof}

The sets $\bP Q_n (W(\pi_n))$ form a nested sequence in $W(\pi_0)$, all containing $\x$. Let 
\[
C(\x) = \bigcap_{n} \bP Q_{\x,n}(W(\pi_n)).
\]
Let $\mu$ be a probability measure on the disjoint union $I_+ \sqcup I_-$ invariant under the non-classical exchange $\x$. Let $(\lambda^\mu_\alpha)$ be the widths assigned by $\mu$ to the bands. \cite[Proposition 4.2]{Gad} shows that if $\x$ is minimal i.e., orbits of $\x$ are dense, then the map $\mu \to (\lambda^\mu_\alpha)$ is a linear homeomorphism from the set of $\x$-invariant probability measures onto the set $C(\x)$. 

Let $\mathbf{m}$ be the Lebesgue measure on each configuration space induced by the $(d-2)$-volume form on it as a codimension 1 submanifold of $\Delta$, normalized so that the total volume of the configuration space is 1. 
 
To estimate the measure of subsets of $\bP Q_n(W(\pi_n))$, we compare the push-forward $(\bP Q_n)_\ast (\mathbf{m})$ measure from $W(\pi_n)$ to $W(\pi_0)$ to the measure $\mathbf{m}$ on $W(\pi_0)$. The Radon-Nikodym derivative of $\mathbf{m}$ with respect to $(\bP Q_n)_\ast (\mathbf{m})$ is the Jacobian $\J(\bP Q_n)$ of the restriction $\bP Q_n: W(\pi_n) \to W(\pi_0)$. Integrating $\J(\bP Q_n)$ over the subset gives its measure. Thus, to give quantitative estimates, one needs to understand $\J(\bP Q_n)$ better.

Suppose $\pi_n$ is the same as $\pi_0$ at some stage $\jmath: \pi_0 \to \cdots \to \pi_n = \pi_0$, and suppose $\kappa$ is a finite splitting sequence starting from $\pi_0$. If $\J(\bP Q_\jmath)$ is roughly the same at all points, then the relative probability that $\kappa$ follows $\jmath$ is roughly equal to the probability that an expansion starts with $\kappa$. We make this precise below:
\begin{definition}\label{C-distortion}
Suppose $\jmath: \pi_0 \to \pi_1 \to \dotsc \to \pi_n$ is a stage and $Q_\jmath$ the associated matrix. For $C > 1$, we say that $\jmath$ is {\em $C$-uniformly distorted} if for all $\y, \y' \in W(\pi_n)$
\[
\frac{1}{C} \leqslant \frac{\J(\bP Q_\jmath)(\y)}{\J(\bP Q_\jmath)(\y')} \leqslant C.
\]
\end{definition}

\begin{remark}\label{distort-distribute}
The matrix $Q_\jmath$ is $C$-distributed if for all $\alpha, \beta \in \A$, the columns of $Q_\jmath$ satisfy 
\[
\frac{1}{C} < \frac{\vert Q_\jmath(\alpha) \vert}{ \vert Q_\jmath(\beta) \vert} < C.
\]
As shown by the analysis of $\J(\bP Q_\jmath)$ in Section 8 of \cite{Gad}, when the columns of $Q_\jmath$ are $C^{1/(d-1)}$-distributed then $\jmath$ is $C$-uniformly distorted. The relative probability statement then becomes:
\end{remark} 

\begin{lemma}\label{rel-prob}
Suppose $\jmath: \pi_0 \to \cdots \to \pi_n = \pi_0$ is $C$-distributed, and let $\kappa$ be a finite splitting sequence starting from $\pi_0$. Let $\jmath \ast \kappa$ denote the sequence $\jmath$ followed by $\kappa$. Then, there exists constant $c>1$ that depends only on $(C, d)$ such that
\[
\frac{1}{c} \mathbf{m}\left( \bP Q_\kappa(W(\pi_\kappa))\right) <  \frac{\mathbf{m}\left(\bP Q_{\jmath \ast \kappa}(W(\pi_{\jmath \ast \kappa}))\right)}{\mathbf{m}\left( \bP Q_\jmath (W(\pi_\jmath))\right)} < c \mathbf{m}\left( \bP Q_\kappa(W(\pi_\kappa))\right).\]
\end{lemma}
\begin{notation}
At any stage $\jmath$, we shall say that two quantities $A_1 \asymp A_2$ if up to a multiplicative constant that depends only on $(C,d)$ and is independent of $\jmath$ they are the same. With this notation:
\[
\frac{\mathbf{m}\left(\bP Q_{\jmath \ast \kappa}(W(\pi_{\jmath \ast \kappa}))\right)}{\mathbf{m}\left( \bP Q_\jmath (W(\pi_\jmath))\right)} \asymp \mathbf{m}\left( \bP Q_\kappa(W(\pi_\kappa))\right).
\]
\end{notation}
\noindent The main technical theorem is \cite[Theorem 1.2]{Gad} which we reproduce below:

\begin{theorem}[Uniform Distortion]\label{Uniform-distortion}
Let $\jmath: \pi_0 \to \pi_1 \to \dotsc \to \pi_n$ be a stage. There exists a constant $C>1$, independent of $\jmath$, such that for almost every $\x \in \bP Q_\jmath(W(\pi_n))$, there is some $m >n$, depending on $\x$, such that $\jmath$ followed by $\pi_n \to \cdots \to \pi_{\x,m}$ is $C$-uniformly distorted. Additionally, $\pi_{\x, m}$ can be arranged to be the same as $\pi_0$.
\end{theorem}

Given Remark~\ref{distort-distribute}, to prove Theorem~\ref{Uniform-distortion} one in fact shows that there is some $m>n$ such that $\jmath$ followed by $\pi_n \to \cdots \to \pi_{\x,m}$ is $C$-distributed. 

A consequence of Theorem~\ref{Uniform-distortion} is \cite[Theorem 11.1]{Gad} which we reproduce below:

\begin{theorem}[Strong Normality]\label{Strong-Normality}
In almost every expansion, for any finite sequence $\jmath$ starting from $\pi_0$, there are infinitely many instances in which $\jmath$ immediately follows a $C$-distributed stage.
\end{theorem}

%%%%%%%%%%%%%%%%%%%%%%%%%%%%%%%%%%%%%%%
\section{Cyclic approximation}\label{cyc approx}
%%%%%%%%%%%%%%%%%%%%%%%%%%%%%%%%%%%%%%%

Let $\ell$ denotes the Lebesgue measure on $I_+ \sqcup I_-$. Let $\pi \in \cG$. Let $\x \in W(\pi)$ be a non-classical exchange.

\begin{theorem}[Cyclic approximation]\label{cyclic}
For almost every set of widths $\x \in W(\pi)$, and any $\delta > 0$ small, 
there is a positive integers $N, n$, a band $\alpha \in \A$, such that for some $J = I(\alpha, n)$
\begin{enumerate}
\item $J \cap \x^k (J) = \varnothing$ for all $1< k < N$.
\item $\x$ is linear on the set $\x^k (J)$ for all $1< k < N$.
\item 
\[
\ell \left( \bigcup_{k = 0}^{N-1} \x^k (J) \right) > 1- \delta.
\]
\item $\ell(J \cap \x^N (J)) > (1 - \delta)\ell(J).$
\end{enumerate}
Moreover, $\alpha$ is orientation preserving in $\pi_n$. 
\end{theorem}

\begin{proof}
 We simply check that the individual steps in the original proof by Veech \cite[Theorem 1.4]{Vee2} hold for non-classical exchanges. 
\subsubsection*{Step 1:} As a consequence of Theorem~\ref{Uniform-distortion}, there is a $C$-distributed stage $\jmath: \pi \to \pi_1 \to \cdots \to \pi_m= \pi'$ such that the matrix $Q_\jmath$ is positive and some band $\alpha$ is orientation preserving in $\pi'$.

\subsubsection*{Step 2:}
Given a constant $0 < \xi < 1$, let $W(\pi', \alpha, \xi)$ be the subset of $W(\pi')$ satisfying $\lambda_\alpha > (1- \xi)  \vert \lambda \vert$. Because $\alpha$ is orientation preserving in $\pi'$, the subset $W(\pi', \alpha, \xi)$ is non-empty and $\mathbf{m}(W(\pi', \alpha, \xi))/ \mathbf{m}(W(\pi')) \asymp \xi^{d-2}$. Consider $W = \bP Q_\jmath(W(\pi', \alpha, \xi))$. By Strong Normality, for almost every $\x \in W(\pi)$ there is a finite splitting sequence $\kappa: \pi \to \cdots \to \pi$ starting and terminating in $\pi$ and depending on $\x$ such that $\bP Q_\kappa(W)$ contains $\x$. Hence, the expansion of $\x$ begins with the concatenation $\kappa * \jmath $ i.e., $\kappa$ followed by $\jmath$. Thus, $Q_{\kappa * \jmath} = Q_\kappa Q_\jmath$. 

\subsubsection*{Step 3:} Let $n$ be the length as a directed path in $\G$ of the splitting sequence $\kappa * \jmath$. Corresponding to the sequence $\kappa * \jmath$, let $\y = \mathcal{R}^n(\x)$ be the exchange induced on corresponding subintervals $I_\pm(n)$ of $I_\pm$. Assuming $\alpha$ is orientation preserving, the widths for $\y$ satisfy
\[
\lambda^{(n)}_\alpha > (1- \xi) \vert \lambda^{(n)} \vert.
\]
In other words,
\[
\ell(I(\alpha, n)) > (1- \xi) \ell(I_+(n) \sqcup I_-(n)).
\] 
Recall from Lemma~\ref{matrix-elements} that the $(\beta, \alpha)$ entry of the matrix $Q_{\kappa * \jmath}$ counts the number of times a point in $I(\alpha, n)$ visits $I(\beta)$ under the exchange $\x$, before returning to $I_+(n) \sqcup I_-(n)$. Hence, for $1< k < q = \vert Q_{\kappa * \jmath} (\alpha) \vert$, we have $I(\alpha, n) \cap \x^k (I(\alpha, n)) = \varnothing$. Finally, $\x^q (I(\alpha, n)) \subset I_+(n) \sqcup I_-(n)$. This implies
\[
\ell \left( I(\alpha, n) \cap \x^q (I(\alpha, n)) \right) > (1-2\xi) \ell \left( I_+(n) \sqcup I_-(n) \right).
\] 
\subsubsection*{Step 4:} 
For a positive $d \times d$ matrix $Q$, let
\[
\nu(Q) = \max_{1 \leqslant i, j, k \leqslant d} \frac{Q_{ij}}{Q_{ik}}.
\]
Recalling inequalities 3.1 and 3.2 from \cite{Vee2}, we get
\begin{eqnarray*}
\vert Q_{\kappa * \jmath}(\beta) \vert &\leqslant& q \nu(Q_{\kappa * \jmath}).\\ \nu(Q_{\kappa * \jmath}) &\leqslant& \nu(Q_\jmath).
\end{eqnarray*}
Let $(\lambda_\gamma) = \x$ be the original widths. Assuming $\alpha$ is orientation preserving, we use $\lambda = Q_{\kappa * \jmath} \lambda^{(n)}$ to get
\begin{eqnarray*}
\vert \lambda \vert  - \lambda^{(n)}_\alpha q &=& \sum_{\beta \neq \alpha} \lambda^{(n)}_\beta \vert Q_{\kappa * \jmath}(\beta) \vert \\
&\leqslant& q \nu(Q_{\kappa * \jmath}) \sum_{\beta \neq \alpha} \lambda^{(n)}_\beta\\
&<& q\nu(Q_{\kappa * \jmath}) \frac{\xi}{1-\xi} \lambda^{(n)}_\alpha \\
&<& \nu(Q_{\kappa * \jmath})\frac{\xi}{1-\xi} \vert \lambda \vert.
\end{eqnarray*}
Rearranging the inequality above, we get
\[
\ell \left( \bigcup_{k=0}^{q-1} \x^k (I(\alpha, n)) \right) > \left( 1- \nu(Q_{\jmath * \kappa} ) \frac{\xi}{1- \xi} \right) \vert \lambda \vert.
\]
\subsubsection*{Step 5:} 
Finally, choose $\xi > 0$ small enough such that $2 \xi < \delta$ and $\nu(Q_{\kappa * \jmath})( \delta/1-\delta) < \xi$. Then, properties (1)-(4) in the theorem hold with $N= q$ and $J= I(\alpha,n)$. 
\end{proof}

\begin{remark}
A main point in the proof above is that $\alpha$ is orientation preserving in $\pi' = \cR^n(\pi)$. This implies that $J = I(\alpha, n)$ has a component each in $I_\pm$. This shall turn out to be relevant later in Section 7.
\end{remark}

%%%%%%%%%%%%%%%%%%%%%%%%%%%%%%%%%%%%%%%
\section{Rigidity sequences}
%%%%%%%%%%%%%%%%%%%%%%%%%%%%%%%%%%%%%%%

\noindent Following \cite{Cha},  a positive integer $n$ is a $\xi$-{\em rigidity time} for a non-classical exchange $\x \in W(\pi)$ if 
\[
\int \limits_{I_+ \sqcup I_-} \vert \x^n(t) - t \vert d\ell < \xi
\]
In particular, a consequence of Theorem~\ref{cyclic} is that for every $\xi>0$ and for almost every $\x \in W(\pi)$ there is a $\xi$-rigidity time. 

A sequence of positive integers $n_1, n_2, \cdots$ is a {\em rigidity sequence} for a non-classical exchange $\x$ if 
\[
\int \limits_{I_+ \sqcup I_-} \vert \x^{n_i}(t) - t \vert d\ell  \to 0
\]
For a sequence of natural numbers $A$, let $a(n)$ be the cardinality of $A \cap \{ 1, \cdots , n\}$. The sequence $A$ is {\em density} 1 if $\lim a(n)/n = 1$. 

To prove disjointness it suffices to show \cite[Remark 9]{Cha} that 
\begin{enumerate}
\item Any sequence of density 1 contains a rigidity sequence for almost every non-classical exchange.
\item For any $\alpha$ we have $\bf{m} (\{\x \in W(\pi): e^{2 \pi i \alpha} \text{ is an eigenvalue of }\x\})=0$.
\end{enumerate}
\begin{theorem}
Let $A$ be a sequence of natural numbers with density 1. Almost every non-classical exchange $\x \in W(\pi)$ has a rigidity sequence in $A$.
\end{theorem}
The above theorem follows from:
\begin{theorem}\label{rigid}For every $\xi>0$ there exists $d(\xi)<1$ such that any sequence of density $d(\xi)$ contains a rigidity sequence for a set of non-classical exchange of measure $1-\xi$.
\end{theorem}
As a corollary to Theorem \ref{rigid}, we get
\begin{corollary}\label{irrat eigen} For any irrational $\alpha$ we have $\bf{m} (\{\x \in W(\pi): e^{2 \pi i \alpha} \text{ is an eigenvalue of }\x\})=0$.
\end{corollary}
\begin{proof}[Corollary \ref{irrat eigen}]
Following exactly the proof of \cite[Corollary 5]{Cha}, we combine Theorem \ref{rigid} with the total ergodicity of almost every non-classical exchange (Theorem \ref{tot-erg}). 

Since almost every non-classical exchange is totally ergodic, $\alpha$ cannot be rational. If $e^{2 \pi i \alpha}$ is an eigenvalue for some irrational $\alpha$ then rotation by $\alpha$ is a factor of $\x$. Then rigidity sequences of $\x$ are also rigidity sequences for the rotation. For any $e >0$, it is possible to construct a sequence of density at least $1-e$ which contains no rigidity sequence for rotation by $\alpha$. See the discussion in \cite{Cha} following Corollary 5. This implies that for any $e>0$, $\ell \left(\{ \x: \alpha \text{ is eigenvalue of } \x\}\right) < e$ finishing the proof.  
\end{proof}
\begin{proof}[Theorem \ref{rigid}]
The proof of Theorem \ref{rigid} is identical to the proof of \cite[Corollary 3]{Cha} except that the following facts need to be verified for non-classical exchanges. This is done using results from \cite{Gad}.

\subsubsection*{Fact 1:} At any stage $\jmath$ in the expansion (with generalized permutation $\pi'$), there are constants $K, C>1$ and probability $0 < p < 1$ independent of $\jmath$ such that there are future stages $\kappa_r$ that are $C$-distributed with $\max_\alpha \vert Q_{\kappa_r}(\alpha)  \vert < K \max_\alpha  \vert Q_\jmath (\alpha) \vert$
and 
\[
\sum \mathbf{m}\left( \bP Q_{\kappa_r}(W(\pi))\right) \geqslant p \mathbf{m} \left( \bP Q_{\jmath}(W(\pi')\right) 
\]
This is \cite[Proposition 10.21]{Gad}. Let $P_i = [K^i, K^{i+1}]$. A consequence of the previous fact is 
\begin{lemma}\label{den-equi}
For almost every $\x \in W(\pi)$ the set of $i$ for which for which some $C$-distributed stage satisfies $\max_{\beta \in \A} \vert Q(\beta) \vert \in P_i$ has positive density i.e., \cite[Lemma 6]{Cha} holds for non-classical exchanges.
\end{lemma}
\begin{proof}[Lemma \ref{den-equi}]
We define a random variable $F_k$ as follows: Let $\jmath_k$ and $\jmath_{k+1}$ be the $k$-th and $(k+1)$-th instances of $C$-distribution in the expansion of $\x$ with corresponding matrices $Q_k$ and $Q_{k+1}$. If
\[
\frac{\max_{\beta \in \A} \vert Q_{k+1}(\beta) \vert }{\max_{\beta \in \A} \vert Q_k(\beta) \vert} \in P_i
\]
then set $F_k(\x) = i$.  Using Fact 1, we see that $F_k$ is at worst exponentially distributed. More precisely, for $a \geqslant 2$, we have $\text{Prob}(F_k = 1) = \ell(\x: F_k(\x) = a) \leqslant (1-p)^{a-1}$. 

Also by Fact 1 we have an estimate on conditional probabilities. That is, for each $(a_1,...,a_n) \in \mathbb{N}^n$ we have that 
\[
\text{Prob}(F_{n+1}=a|(F_1,...,F_n)=(a_1,...,a_n))\leqslant (1-p)^{a-1}.
\]
It follows that for almost every $\x$ we have 
\[
\lim_{n \to \infty}\frac{1}{n} \sum_{k =1}^n F_k(\x) \leqslant \text{Prob}(F_k=1) + \sum_{a \geqslant 2} a(1-p)^{a-1} < \infty
\]
It follows from the limit above that the set of $i$ such that the $C$-distributed stages satisfy $\max_{\beta \in \A} \vert Q(\beta) \vert \in P_i$ has positive density.
\end{proof}

\subsubsection*{Fact 2:}
The probability that a sequence $\kappa$ follows a $C$-distributed stage $\jmath$ is roughly the same as the probability that a sequence begins with $\kappa$. Suppose $\kappa$ terminates in the generalized permutation $\pi'$. Quantitatively, what is needed in the proof in \cite{Cha} is an estimate of the form:
\[
\mathbf{m}\left(\bP Q_\kappa(W(\pi'))\right) < \frac{\mathbf{m}\left(\bP Q_{\jmath \ast \kappa}(W(\pi'))\right)}{\mathbf{m}\left(\bP Q_\jmath(W(\pi))\right)} C^{d-1}
\]
which simply follows from Remark~\ref{distort-distribute}. 

\subsubsection*{Fact 3:} 
Let $R> 1$ be a real number. We need an estimate for the number of maximal columns with norm $R$ i.e., columns $Q(\alpha)$ such that $\vert Q(\alpha) \vert =  \max_{\beta \in \A} \vert Q(\beta) \vert = R$ over the set of all possible $C$-distributed stages in the expansions of non-classical exchanges. A simple count shows that the number of vectors $v$ for which $\vert v \vert \asymp R$ is $O(R^{d-1})$. However, for non-classical exchanges, there is a restriction coming from the fact that the projective linear maps must map configuration spaces to configuration spaces. Thus, not every $v$ with $\vert v \vert \asymp R$ can be a maximal column $Q(\alpha)$. \cite[Lemma 12.2]{Gad} implies that for $C$-distributed stages, the contraction of $\bP Q$ is uniform in all directions. To be precise, \cite[Lemma 12.2]{Gad} implies that the projectivization $\bP Q(\alpha)$ of a maximal column $Q(\alpha)$ has to lie in a $C^2/R^d$-neighborhood of $W(\pi)$ in $\Delta$. This gives the estimate for the number of maximal columns with $R$ to be $O(R^{d-2})$.

\subsubsection*{Fact 4:} 
In the same situation as Fact 3, we want to estimate the $\mathbf{m}$-measure of the set of non-classical exchanges given by a $C$-distributed stage i.e., for a $C$-distributed stage terminating in a generalized permutation $\pi'$ we want to estimate $\mathbf{m}(\bP Q(W(\pi'))$ in terms of $R$. In \cite[Section 8]{Gad}, we analyze the Jacobian of the restriction to configuration spaces of a projective linear map $\bP Q$. Precisely, we show that  
\[
\J(\bP Q) (\y) = \frac{1}{a \vert Q \y \vert^{d-1}} 
\]
where $a> 0$ is a constant that depends on the stage. \cite[Lemma 12.2]{Gad} then implies that for any $C$-distributed stage, the constant $a$ has a lower bound that depends only on $C$. Consequently,
\[
\mathbf{m}(\bP Q(W(\pi')) \asymp O(R^{-(d-1)})
\]
Facts 3 and 4 together imply that the set of $\x$ for which there is a stage $\kappa$ in the Rauzy expansion of $\x$ such that $\max_{\beta \in \A} \vert Q_\kappa (\beta) \vert \asymp R$ is at most $O(R^{-1})$ i.e., 
\cite[Lemma 10]{Cha} holds for non-classical exchanges. Along with the estimate $\mathbf{m}(W(\pi, \alpha, \xi)) \asymp \xi^{d-2}$ for the set $W(\pi, \alpha, \xi)$ defined in the proof of Theorem \ref{cyclic}, the above lemma implies that $\mathbf{m}(\x \in W(\pi) : R \text{ is expected } \xi\text{-rigidity time for } \x) \asymp \xi^{d-2}/R$.

\subsubsection*{Fact 5:} 
From Lemma \ref{den-equi} and Fact 2 it follows that the set of $i$ for which $\x$ has a $\xi$-rigidity time with fixed previous induction steps in $[K^i,K^{i+1}]$ has measure at least $c \xi^{d-2}$.

The estimates at the end of Facts 4 and 5 prove Theorem \ref{rigid} by the exact argument in the proof of \cite[Corollary 3]{Cha}. The key point is that the estimates establish that there is a constant $C'$ independent of $\xi$ so that among the set of nonclassical exchanges that have a $\xi$ rigidity time between $K^i$ and $K^{i+1}$, for the special reason given above, the set of those for which this time is in a set of density $\delta$ in $[K^i,K^{i+1}]$ has proportion at most $C'\delta$. This follows from Fact 3 which limits how many $C$-distributed matrices can have the largest column sum $R$ for a single $R$.
\end{proof}

\section{Total ergodicity}

A measure preserving transformation is said to be {\em totally ergodic} if every forward iterate of it is ergodic. To finish the proof of Theorem~\ref{disjoint}, we show:

\begin{theorem}\label{tot-erg}
If $\pi$ has an orientation preserving band then almost every $\x$ in $W(\pi)$ is totally ergodic. 
\end{theorem}

We first give a sufficient condition for total ergodicity.
\begin{proposition}\label{suff cond} 
A non-classical exchange $\x$ is totally ergodic with respect to $\ell$ if for any prime $p$, there are arbitrarily good cyclic approximations of height $n$ coprime to $p$.
\end{proposition}
\begin{lemma}\label{tot-prime}
Suppose $\x$ is ergodic but not totally ergodic. There is a prime $p$ such that $\x^p$ is not ergodic. 
\end{lemma}
\begin{proof}[Lemma \ref{tot-prime}]
Suppose $\x^k$ is not ergodic. By the ergodic decomposition theorem, there exists a $\x^k$ invariant set $S \subset I_+ \sqcup I_-$ with $\ell(S)>0$ such that $S$ is the support of an ergodic measure $\mu$ of $\x^k$, and any other ergodic measure of $\x^k$ assigns measure 0 to $S$. Any translate $\x^n S$ is $\x^k$ invariant and carries the $\x^k$ ergodic measure $(\x^n)_\ast \mu$. Hence, by the decomposition theorem, either  $\x^n S$ and $S$ coincide or their intersection is empty i.e. either $\x^n S \cap S = S$ or $\x^n S \cap S = \varnothing$. Let $m = \min \{n : \ell(\x^n S \cap S) > 0 \}$.  Then $\x^m S = S$. If $p \vert m$ then the set
\[
B = S \cup \x^{m/p} S \cup \x^{2(m/p)} S \cup \cdots \cup \x^{(p-1)(m/p)} S
\]
is $\x^p$ invariant and does not have full measure. Hence, $\x^p$ is not ergodic.
\end{proof}

Before the formal proof of Proposition \ref{suff cond} we state the idea, which we make precise by using Lusin's Theorem. If $\x$ is ergodic but not totally ergodic then it has a factor $F:\mathbb{Z}/p\mathbb{Z}\to \mathbb{Z}/p\mathbb{Z}$ by $F(i)=i+1$ given by identifying points with the ergodic component of $\x^p$ they are in. By assumption there exist $n_i$ such that $p \nmid n_i$ and $\x^{n_i}$ are arbitrarily close to the identity. Because $\x^{n_i}$ is arbitrarily close to the identity and $F$ is a factor $F^{n_i}$ is close to the identity too. But $F^{n_i}=F^t$ where $t \equiv n_i $ mod $p$.  Because $gcd(p,n_i)=1$ we have $F^{n_i}$ is not close to the identity.

\begin{proof}[Proposition \ref{suff cond}]
Let $B$ be a $\x^p$-invariant set that is the support of an ergodic measure as above. Then the translates $\{B, \x B , \x^2 B, \cdots, \x^{p-1} B\}$ are all disjoint. 

By Lusin's Theorem, we may assume that $B$ contains all but a very small proportion, say $1/u$ for very large $u$, of a subinterval $V$ of $I_+ \sqcup I_-$ of size $\delta_1$. Choose $\delta$ for the cyclic approximation to be less than $\delta_1/p$.  By assumption, we may chose
 the height $n$ of a cyclic approximation to $\x$ to be such that gcd$(n,p)=1$ and $n>3/\delta$. For $0 \leqslant j < n$, call the sets $\x^j J$ as {\em levels} of the approximation. By the choice of $\delta$, there is some level $\x^s J$ such that at least one of the two subintervals of $\x^s J$, say $\x^s J_+$ is in $V$. Consider the subset 
 \[
 J'_+= \underset{i=0}{\overset{p}{\cap}}\x^{-in}(J_+).
 \] 
These are points in $J_+$ whose first $pn$ iterates of $\x$ lie in the corresponding level mod $n$. By the approximation theorem, we have the estimate:
\[
\ell(J'_+) > \ell(J_+)(1-p \delta) > \ell(J_+)(1-\delta_1)
\]
Notice that if $t \in J'_+$ then the sequence $(\x^{ip}) t$ for $0 \leqslant i < n$ hits each level exactly once. So we get the estimate:
 \[
 \ell\left(\underset{j=0}{\overset{p-1}{\cap}}\x^{-jn}(\x^s J_+\cap B)\right)> \ell(J_+)(1-\delta_1-1/u).
 \] 
This implies that $\underset{i=0}{\overset{n-1}{\cup}}\x^{pi} \left( \underset{j=1}{\overset{p-1}{\cap}}\x^{-jn}(\x^sJ_+ \cap B)\right)$ is most of the set $\underset{i=0}{\overset{n-1}{\cup}}\x^{-i}(J_+)$. By $\x^p$ invariance of $B$, the former is also a subset of $B$. Since 
\[
\ell\left(\underset{i=0}{\overset{n-1}{\cup}}\x^{-j}(J_+) \right) > 1/2- \delta,
\]
we must have $\ell(B) \geqslant 1/2$. If $\ell(B) \neq 1$ then disjointness of the translates of $B$ implies $\ell(B) = 1/2$. This means that $p=2$ and $\x B = B^c = I_+ \sqcup I_- \setminus B$. This contradicts the fact that the subset of $B$ constructed above contains most of $\underset{i=0}{\overset{n-1}{\cup}}\x^{-i}(J_+)$. Hence, $B$ has full measure, and $\x^p$ is ergodic with respect to $\ell$. 
\end{proof}

It remains to show the existence of arbitrarily good cyclic approximations with heights coprime to $p$. This proposition is the only result that requires the hypothesis that $\pi$ has orientation preserving bands. 
\begin{proposition} \label{prime approx} Suppose $\pi$ has an orientation preserving band. For any $\delta>0$ and prime $p$, the set
\[
\{ \x \in W(\pi): \exists n \text{ with } \gcd(n,p) = 1 \text{ s.t. } \x \text{ has a height }n \text{ cyclic approximation with constant } \delta \}
\]
has full measure.
\end{proposition}

We first state and prove a well known fact for completeness.
\begin{lemma}  \label{sl prime} Let $Q$ be a matrix in $SL(d,\Z)$ with column norms $q_1, q_2, \cdots, q_d$. Then, $\text{gcd}(q_1,q_2,...,q_d)=1$.
\end{lemma}
\begin{proof}[Lemma \ref{sl prime}]
Let $A$ be the $d\times d$ matrix with all entries 1. Every entry of $AQ$ is divisible by $\text{gcd}(q_1,q_2,...,q_d)$. Therefore, for all matrices $B  \in SL(d, \Z)$ the matrix $AQB$ has every entry divisible by gcd$(q_1,q_2,...,q_d)$. In particular, this is true for $B= Q^{-1}$. However $AQQ^{-1}=A$. Therefore, gcd$(q_1,q_2,...,q_d)=1$.
\end{proof}

\begin{lemma} \label{sl prime 2}
Let $p$ be a prime and suppose $\pi$ has orientation preserving bands. Then, there is a constant $C>1$ such that for almost every $\x \in W(\pi)$ there are infinitely many $\{n_k\}_{k=1}^\infty \subset \N$ such that $Q_{\x, n_k}$ is $C$-distributed and there is an orientation preserving band $\alpha_k$ for which $\vert Q_{\x, n_k}(\alpha_k) \vert$ is coprime to $p$. 
\end{lemma}

\begin{claim}
Let $\pi \to \pi_1 \to \pi_2 \cdots$ be any infinite expansion. With the hypothesis that $\pi$ has orientation preserving bands, for every $n \in \N$ either some orientation preserving band in $\pi_n$ has column norm in $Q_n$ coprime to $p$ or if column norms in $Q_n$ of all orientation preserving bands in $\pi_n$ are divisible by $p$ then there exists orientation reversing bands $\beta_1, \beta_2$ on opposite sides i.e., on $I_+$ and $I_-$ respectively, such that $\vert Q_n(\beta_1) \vert$ and $\vert Q_n(\beta_2) \vert$ have remainders $r_1$ and $r_2$ such that $r_1+ r_2 \neq 0 \mod p$.
\end{claim}
\begin{proof}
When $n=0$, all column norms are 1, so the claim is true. Suppose the claim is true for $n$. 

First, suppose that $\vert Q_n(\alpha)\vert$ is coprime to $p$ for some orientation preserving band $\alpha$ in $\pi_n$. There remains an orientation preserving band with column norm in $Q_{n+1}$ coprime to $p$ unless $\alpha$ is the only orientation preserving band in $\pi_n$ with $\vert Q_n(\alpha) \vert \neq 0$ mod $p$ and $\alpha$ splits some orientation reversing band $\beta$ in the subsequent split $\pi_n \to \pi_{n+1}$. Lets suppose $\alpha$ splits the orientation reversing band $\beta$ and let $r_2 = \vert Q_n(\beta) \vert$ mod $p$. Then $\vert Q_{n+1}(\alpha) \vert = r_1 + r_2$ mod $p$. Consider an orientation reversing band $\gamma$ on the opposite side and let $r_3 = \vert Q_{n+1}(\gamma) \vert$ mod $p$. Then $r_3+ r_2 = 0$ mod $p$ and $r_3 + r_1 + r_2 = 0$ mod $p$ cannot be simultaneously true. Thus, at least one of the pairs $(\gamma, \alpha)$ and $(\gamma, \beta)$ satisfy the remainder condition of the claim.

Now suppose that all orientation preserving bands in $\pi_n$ have column norms divisible by $p$, and there are orientation reversing bands $\beta_1$ and $\beta_2$ on $I_+$ and $I_-$ respectively such that $\vert Q_n(\beta_1) \vert$ and $\vert Q_n(\beta_2) \vert$ have remainders $r_1$ and $r_2$ satisfying $r_1+ r_2 \neq 0$ mod $p$. Without loss of generality, we may assume that at least one of $\beta_1$ or $\beta_2$, say $\beta_1$, is in a critical position. Let $\gamma$ be the band in the other critical position. If $\gamma$ splits $\beta_1$ then the $\beta_1$-column is unchanged, so assume that $\beta_1$ splits $\gamma$. If $\gamma$ is orientation preserving then $\beta_1$ remains orientation reversing on $I_+$ and $\vert Q_{n+1}(\beta_1) \vert = r_1$ mod $p$. So the pair $(\beta_1,\beta_2)$ continues to satisfy the remainder condition. If $\gamma$ is orientation reversing then either $\vert Q_{n+1}(\beta_1) \vert \neq 0$ mod $p$ in which case we have $\beta_1$ as an orientation preserving band with column norm coprime to $p$, or $\vert Q_{n}(\gamma) \vert = - r_1$ mod $p$. In the later case, there must exist an orientation reversing band $\beta_3$ on $I_+$, and suppose $\vert Q_{n+1}(\beta_3)  \vert = r_3$ mod $p$. Then $r_3+ r_2 = 0$ mod $p$ and $r_3 - r_1 = 0$ mod $p$ cannot be simultaneously true. Thus, at least one of the pairs $(\beta_3, \beta_2)$ and $(\beta_3, \gamma)$ satisfies the remainder condition.

The claim follows by induction.
\end{proof}

Now restrict to $\pi$ and consider an assignment of remainders mod $p$ to all bands such the remainders assigned to orientation preserving bands are zero and there is a pair $(\beta_1, \beta_2)$ of orientation reversing bands on $I_+$ and $I_-$ respectively with remainders $(r_1, r_2)$ such that $r_1+ r_2 \neq 0$ mod $p$. Starting from such an assignment, we can derive an assignment of remainders in all stages obtained from $\pi$. 
\begin{claim}\label{orp seqn}
Given an initial assignment of remainders mod $p$ to bands in $\pi$ as described above, there is a splitting sequence from $\pi$ such that in the final stage there is an orientation preserving band with nonzero remainder mod $p$. 
\end{claim}
\begin{proof}
By \cite[Proposition 10.1]{Gad}, there is a shortest splitting sequence that brings one of $\beta_1$ or $\beta_2$ to the critical position on the corresponding interval $I_+$ or $I_-$. Without loss of generality, say $\beta_1$ is  brought to the critical position on $I_+$. Then the $\beta_1$ and $\beta_2$ remainders remain unchanged during this sequence. We may also assume that the conclusion of the claim is not satisfied during this sequence. After $\beta_1$ is in the critical position on $I_+$ split it by every band to the left of $\beta_2$ on $I_-$ till $\beta_2$ appears in the critical position on $I_-$. Again, the $\beta_1$ and $\beta_2$ remainders remain unchanged, and we might assume that the conclusion of the claim is not satisfied during this sequence. But then, in the next split either $\beta_1$ or $\beta_2$ becomes orientation preserving and has the non-zero remainder $r_1+ r_2$ mod $p$, proving the claim.
\end{proof}

\begin{proof}[Lemma~\ref{sl prime 2}]
For each assignment $\sigma$ as above of remainders to bands in $\pi$, Claim~\ref{orp seqn} gives a splitting sequence $\jmath(\sigma)$ from $\pi$ that produces an orientation preserving band with non-zero remainder mod $p$. Since the number of assignments $\sigma$ is finite (bounded above by $p^d$), there is a bound on the lengths of the sequences $\jmath(\sigma)$. 

By Theorem \ref{Uniform-distortion}, there is a constant $C>1$ such that for almost every $\x \in W(\pi)$ there are infinitely many $\{m_k\}_{k=1}^\infty$ such that $Q_{\x, m_k}$ is $C$-distributed and $\pi_{\x, m_k} = \pi$. If the column norms in $Q_{\x, m_k}$ of all orientation preserving bands in $\pi$ are divisible by $p$ then consider the remainders mod $p$ of the column norms of all bands. Call this assignment $\sigma_k$. As established earlier, there must be a pair of orientation reversing bands $(\beta_1, \beta_2)$ with remainders $(r_1, r_2)$ such that $r_1+ r_2 \neq $ mod $p$. In this case, consider $\pi \to \cdots \pi_{\x, m_k}$ followed by $\jmath(\sigma_k)$. Since the length of $\jmath(\sigma_k)$ is bounded above independent of $\x$ and $m_k$, there is a lower bound independent of $\x$ and $m_k$ on the relative probability that $\jmath(\sigma_k)$ follows $\pi \to \cdots \to \pi_{\x, m_k}$. For the same reason, there is also an upper bound independent of $\x$ and $m_k$ on the distortion introduced by $\jmath(\sigma_k)$. 

By Strong Normality (Theorem~\ref{Strong-Normality}), for almost every $\x$, there is a subsequence of $\{m_k\}$, call it $\{n_k\}$, such that either the column norm in $Q_{\x, n_k}$ of some orientation preserving band is coprime to $p$ or the sequence $\jmath(\sigma_k)$ follows $\pi \to \cdots \to \pi_{\x, n_k}$ in the expansion for $\x$, resulting in an orientation preserving band with column norm coprime to $p$. Given that $\jmath(\sigma_k)$ introduces a distortion bounded above independently of $\x$ and $n_k$, we can choose a bigger constant $C$ such that all of these stages are $C$-distributed. This proves Lemma~\ref{sl prime 2}.
\end{proof}

\begin{proof}[Proposition \ref{prime approx}] For almost every $\x$, there infinitely many $\{n_k\}_{k=1}^{\infty}\subset \mathbb{N}$ such that the matrix $Q_{\x,n_k}$ is $C$-distributed. By Lemmas~\ref{sl prime} and~\ref{sl prime 2}, there is $\alpha_k$ orientation preserving in $\pi_{\x, n_k}$ such that $\vert Q_{\x, n_k} (\alpha_k) \vert$ is coprime to $p$. Let $X_k$ denote the set of $\x$ such that $\mathcal{R}^{n_k}(\x) \in W(\pi,\alpha_k,\delta/C)$. From the proof of Theorem~\ref{cyclic}, the exchanges $\x \in X_k$ have a cyclic approximation of height $\vert Q_{\x, n_k}(\alpha_k) \vert$ with constant $\delta$. 

Let  $\mu$ be the minimum over $\alpha \in \A$ of $\mathbf{m}( W(\pi, \alpha, \delta/C))$. By Lemma~\ref{rel-prob}, up to a universal multiplicative constant, we have $\mathbf{m}(X_k) \geqslant \mu$. Thus, $\sum_k \mathbf{m} (X_k) = \infty$. 
Moreover, since the sets $X_k$ are $C$-distributed we have 
$\mathbf{m}(X_k\cap X_L)\leqslant C\mathbf{m}(X_k)\mathbf{m}(X_l)$, and so almost every $\x$ belongs to some $X_k$, and in fact, belongs to infinitely many $X_k$. 
\end{proof}
\begin{proof}[Theorem \ref{tot-erg}] By Proposition \ref{suff cond} and the fact that a countable intersection of full measure sets has full measure it suffices to prove Proposition \ref{prime approx}.
\end{proof}

\begin{remark}
Starting from a generalized permutation with all bands orientation reversing, in all later stages, the column norms are even for all orientation preserving bands and odd for all orientation reversing bands. Thus, all cyclic approximations have even height. This is in line with the fact that $\x^2$ leaves $I_+$ and $I_-$ each invariant. In fact, Proposition~\ref{suff cond} shows that these are the only invariant subsets for $\x^2$, and also that all odd iterates of $\x$ are ergodic. 
\end{remark}

\section{Proof of Theorem~\ref{disjoint}}
\noindent The proof of Theorem~\ref{disjoint} now follows verbatim as in \cite{Cha}. The results of the preceding sections establish the necessary facts for non-classical exchanges.
In particular, we use the criterion for generic disjointness \cite[Remark 9]{Cha}. It suffices to show that 
\begin{enumerate}
\item For any $\alpha$ almost every $\x$ does not have $\alpha$ as an eigenvalue.
\item For any $A \subset \mathbb{N}$ with density 1 almost every $\x$ has a rigidity sequence in $A$.
\end{enumerate}
Corollary \ref{irrat eigen} and Theorem \ref{tot-erg} establish condition 1. Theorem \ref{rigid} establishes condition 2.

\appendix
\section*{Appendix}
In this appendix, we prove disjointness of vertical flows for quadratic differentials, namely Theorem \ref{flow disjoint}. Theorem \ref{flow disjoint} is easier to prove than Theorem \ref{disjoint} by using the mixing of \teichmuller geodesic flow, namely Theorem \ref{M-V}.

We first set out some preliminaries. A quadratic differential $q$ on a Riemann surface $X$ defines by contour integration a singular flat metric on $X$ exhibiting it as a half-translation surface: the surface $X$ carries a set of charts to $\mathbb{C} = \mathbb{R}^2$ with transition functions of the form $z \to \pm z + a$. We restrict to unit area quadratic differentials i.e., the area of the surface with the induced flat metric is 1. Conversely, a singular flat metric on a surface $S$ with the above properties defines not only a conformal structure $X$ on it but also a quadratic differential on $X$. In summary, these turn out to be equivalent notions. We will denote the half-translation surface associated to $q$ by $X(q)$. We denote the space of unit area quadratic differentials by $\mathcal Q$. The group $SL(2, \R)$ has a natural action on $\mathcal Q$ as affine transformations of $X(q)$. The action of the diagonal part
\[
g_t = \left[ \begin{array}{cc} e^t & 0 \\ 0 & e^{-t} \end{array} \right]
\]
is called \teichmuller geodesic flow. The Masur-Veech measure $\mathbf{m}$ is the $g_t$-invariant Liouville measure on $\mathcal{Q}$. 

A choice of a basis for the relative homology of the surface defines by integration (holonomy or period) local coordinates on $\mathcal{Q}$ and the Masur-Veech measure is the Lebesgue measure class in these co-ordinates. See \cite{Mas2} for  a detailed discussion of quadratic differentials.

\begin{theorem}\label{M-V}(Masur-Veech \cite{Mas}, \cite{Vee})The \teichmuller geodesic flow $g_t$ is mixing for the Masur-Veech measure.
\end{theorem}
The proof of Theorem \ref{flow disjoint} again follows the strategy in outlined in \cite[Remark 9]{Cha} namely to show that any sequence of density 1 contains a rigidity sequence for almost every transformation and for any $\alpha \neq 1$, the set of transformations with $\alpha$ as an eigenvalue has measure zero.

Rigidity sequences for a vertical flow have many equivalent definitions. For the purposes of this appendix we will require an approximate version of $\epsilon$ rigidity times. We choose the following definition: given a quadratic differential $q$ let $d$ denote the flat metric on $X(q)$ and let $dA$ be the natural area form from its charts to $\mathbb{C}$. We define $t$ to be an $\epsilon$ rigidity time for a flow $F$ on the half-translation surface if $\int_{X(q)} d(x, F^t (x))\,dA<\epsilon.$

\begin{proposition}\label{lots rig}Let $A$ be a sequence in $\mathbb{R}$ with positive upper density. Then almost every quadratic differential's vertical flow has a rigidity sequence in $A$.
\end{proposition}
\begin{corollary}\label{eigen zero} Let $\alpha \neq 1$ then 
\[
\mathbf{m}(\{q: \alpha  \text{ is an eigenvalue of the vertical flow on }q\})=0.
\]
\end{corollary}
Since rational eigenvalues are not special for $\mathbb{R}$-actions, Corollary \ref{eigen zero} follows from Proposition \ref{lots rig} analogous to how Corollary \ref{irrat eigen} follows from Theorem \ref{rigid}.

\begin{proof}[Theorem \ref{flow disjoint} assuming these results] Analogous to the proof of Theorem \ref{disjoint}, Theorem \ref{flow disjoint} follows from Proposition \ref{lots rig} and Corollary \ref{eigen zero}.
\end{proof}

It remains to prove Proposition \ref{lots rig}.  A metric cylinder for a quadratic differential $q$ is a Euclidean cylinder in the flat metric on $X(q)$. Roughly speaking, if $X(q)$ has a metric cylinder with period $T$ in direction close to vertical and area close to 1, then the vertical flow has a $O(\epsilon)+2\epsilon \, \text{diam} (X(q))$ rigidity time of about $T$. This is because the vertical trajectories that stay in the almost vertical cylinder return close. The trajectories not in the almost vertical cylinder have small measure and are distance at most $\text{diam} (X(q))$ from their starting point. We make this precise below.

\begin{lemma}\label{key} Suppose $X(q)$ has a cylinder $C$ in direction $\phi$ with area at least  $1-\epsilon$ and period $T\geqslant 1$. Let $\theta$ be a direction such that $|\theta-\phi|<  \epsilon/T^2$. If $F_{\theta}^t(x) $ is inside $C$ for all $0<t<T $ then 
$d(F^t_{\theta}(x),F^t _{\phi}(x))<\epsilon/T$. In particular, $d(F^T_\theta (x), x) < \epsilon/T$. 
\end{lemma}
\begin{proof}
Since a cylinder is convex, there is no singularity along the flat geodesic inside the cylinder $C$ between $F^t_{\theta}(x)$ and $F^t_{\phi}(x)$. Consequently, 

\[d(F^t_{\theta}(x),F^t_{\phi}(x))= 2t  \sin \left(\frac{\vert \theta-\phi \vert}{2}\right) \leqslant t \vert \theta-\phi \vert < \epsilon/T.\] 
\end{proof}
\begin{corollary}If $\epsilon$ is small enough then under the assumption of the previous lemma the flow $F_\theta$ has $\epsilon +3\epsilon  \text{diam}(X)$ rigidity time at $T$.
\end{corollary}
\begin{proof}  We show that if $\epsilon$ is small enough then $T$ is a rigidity time by estimating $\int_{X(q)} d(F_{\theta}^T (x),x)\, d A$. Split the integral $\int_{X(q)} d(F_{\theta}^T (x),x)\, d A$ over the set $W$ of points such that $F^{t}_{\theta}(x)$ is in $C$ for $0<t<T$, and complement $X(q) \setminus W$. By the previous lemma $d(F_{\theta}^T(x),x)< \epsilon/T < \epsilon$ for all $x \in W$ whereas $d(F_{\theta}^T(x), x) \leqslant \text{diam} X(q)$ over $X(q) \setminus W$. It remains to estimate the area of $X(q) \setminus W$. To do this first note that $d(x, \partial C) \geqslant 2T \sin (\vert \theta - \phi \vert/2)$ is a sufficient condition for $x$ to be in $W$. To maximize the area of $X(q) \setminus W$ we maximize the right hand side of the inequality to consider the extremal case when $X(q) \setminus W$ contains all points $x$ such that $d(x, \partial C) \leqslant \epsilon/T$. The cylinder $C$ has circumference $T$ and so the area of the set of points $x \in C$ such that $d(x, \partial C)\leqslant \epsilon/T$ is $\epsilon$. Thus the area of $X(q) \setminus W$ is bounded above by $3 \epsilon$.
\end{proof}

\begin{lemma}\label{lem:key} Let $M>1$.  For any small enough $\epsilon>0$  there exists a constant $c_{M,\epsilon}:=c$ depending on the genus, $M$ and $\epsilon$ such that for any measurable set $U \subset \mathcal{Q}$ for which the diameters of surfaces in $U$ are at most $D$, there is $L>1$ large enough such that 
\[
\mathbf{m}(\{q\in U: \text{ the vertical flow on }q \text{ has a }4D\epsilon \text{ rigidity time in }(L,LM\})\geqslant c\mathbf{m}(U).
\]
\end{lemma}
The proof of rigidity boils down to renormalization dynamics of the \teichmuller flow. The renormalization dynamics are measure preserving and mixing and so we obtain the lemma.

\begin{proof} \textbf{(Renormalization dynamics)}
Since matrices act affinely on half-translation surfaces if $X(g_tq)$ has a metric cylinder of area at least $1-\epsilon$, period $T$, in direction $\theta$ with the vertical foliation then $X(q)$ has a cylinder of area at least $1-\epsilon$ with period $$\sqrt{(e^tT\cos \theta)^2+(e^{-t}T\sin \theta)^2}$$ in direction 
$$\arctan(e^{-2t}\tan\theta)$$  with the vertical. By applying the previous corollary it follows that if  $X(g_tq)$ has a cylinder of area at least $1-\epsilon$, period $T$, in direction within $\epsilon/T^2$ of the vertical then the vertical flow on $X(q)$ has a $\epsilon + 3\epsilon \text{diam} (X(q))$ rigidity time $e^tT$. 

To complete the proof, consider the set of quadratic differentials that have a cylinder with area at least $1-\epsilon$, period in $T \in (1,M)$ and in direction within angle $\epsilon/T^2$ of the vertical. These are open conditions in period coordinates on $\mathcal{Q}$ and hence the set contains some closed ball $B$. As shown above if $g_tq\in B$ then $q$ has a cylinder of period between $(e^t\cos(\epsilon),e^tM(\cos(\epsilon))+e^{-t})$ in direction within $\epsilon/T^2e^{2t}$ of  the vertical. So it has a $4D \epsilon$ rigidity time in $(e^t\cos(\epsilon),e^tM(\cos(\epsilon))+e^{-t}) \asymp (e^t, M e^t)$ where the approximation is justified since the error is $O(\epsilon^2)$. The lemma follows because $g_t$ is measure preserving and mixing and so $c$ can be chosen to be any number smaller than $\mathbf{m}(B)$.
\end{proof}

In subsequent arguments we will apply approximations such as above without comment. It should be noted that the above proof does not require that the periods be restricted to be in $(1, M)$. It goes through so long as the ratio of the largest to the smallest period is bounded by $M$. In other words, the proof works even if $ T \in (T', T'M)$ for some fixed $T'$.

Following the same ideas as Lemma \ref{lem:key} we get:
\begin{corollary}\label{cor:local} Given $K$ a compact part of $\mathcal{Q}$ there exists $c'$ depending only on the maximal diameter of surfaces in $K$ and $\epsilon$ so that if $A \subset \mathbb{R}$ then for any open set $U \subset K$ and all large enough $L$ we have
\[
\mathbf{m}(\{q\in U: \text{ the vertical flow on }q \text{ has an }\epsilon \text{ rigidity time in } A \cap (L,LM\}) >c'\frac{|A\cap (L,ML)|}{(M-1)L}\mathbf{m}(U).
\]
\end{corollary}

To prove the corollary we need the straightforward fact that any ball in $\mathcal{Q}$ contains an appropriate \texttt{"}flowbox\texttt{"} with definite proportion of the measure. Let $B(q, r)$ denote the ball of radius $r$ in the period co-ordinates centered at $q$. 
\begin{lemma}\label{foliate} For any $B(q,r)$, $e>0$ there exists $\delta>0$, $V \subset B(q,r)$ so that 
\begin{enumerate}
\item $g_sV \in B(q,r)$ for all $0\leqslant s\leqslant \delta$
\item $g_sV \cap g_{s'}V=\emptyset$ for all $0\leqslant s < s'\leqslant \delta$
\item $\mathbf{m} \left(B(q,r)\setminus \cup_{s\in (0,\delta)}g_sV\right) < e \mathbf{m}(B(q,r))$
\item $\cup_{s \in (0,r)}g_sV$ is open for all $0<r<\delta$.
\end{enumerate}
\end{lemma}

\begin{proof}[Corollary \ref{cor:local}]
It suffices to prove the corollary for balls. So suppose $U=B(q,r)$. Set $e=1/9$ and let $\delta,V$ be as in Lemma \ref{foliate}. Set $M= 1 + \delta/4$ for the range of periods in Lemma \ref{lem:key} and let $B'$ be the ball in Lemma \ref{lem:key} corresponding to this choice of $M$. The proof proceeds in two steps. First, mixing of the geodesic flow implies that for all large times $t$ a definite proportion of $g_tV$ lands in $B'$. Second, for every $q \in V$ such that $g_tV \in B'$, a subset with definite proportion of the flow line $g_s q$ given by Lemma \ref{foliate} has rigidity times in $A$. The corollary then follows by applying Fubini's theorem. 

To apply mixing, we pass to a thin slice with base $V$. To be precise, for $\delta' < < \delta$ mixing of the geodesic flow implies that for all $t$ large enough we have 
\[
\mathbf{m}\left(B' \cap g_t  \left(\cup_{s \in (0, \delta')} g_s V\right) \right)\geqslant \frac 1 2 \mathbf{m}(B')\mathbf{m}\left(\cup_{s \in (0,\delta')}g_sV\right).
\]
Following Lemma \ref{lem:key}, a quadratic differential $q' \in \cup_{s \in (0, \delta')} g_s V$ such that $g_t q' $ is in $B'$ has a rigidity time at $e^t \text{period}(g_t q')$. With our choice of $M$ this means that $q'$ has a rigidity time between $e^t$ and $e^t(1 + \delta/4)$. Since $g_{t-s} (g_s q') = g_t q'$ it follows that $g_sq'$ has a rigidity time at $e^{t-s}$period$(g_tq)$. So for any $r\in [e^{t-\delta/2},e^t]$ we have that $r$ is a rigidity time for some $g_sq'$ where $s < \delta$. 

Let $\ell$ be the Lebesgue measure on $\mathbb{R}$ and let $\rho$ be the upper density of $A$. Then there exists arbitrarily large $t$ so that 
\begin{equation*}\label{eq:good time}
\ell([e^{t-\delta/2},e^{t}]\cap A)> \frac{\rho}{2} e^{t-{\delta/2}}(e^{\delta/2}-1).
\end{equation*}

For $t$ satisfying both inequalities and $q' \in \cup_{s \in (0, \delta')} g_s V$ such that $g_t q' $ is in $B'$, we consider the set of $s$ such that $g_s q'$ has a rigidity time $r(s) \in [e^{t-\delta/2},e^{t}]\cap A$. Since the ratio of the derivatives of the function $s \to e^{t-s}$ at distinct values of $s$ in $(0, \delta)$ stays bounded between $e^{-\delta}$ and $e^{\delta}$ it implies that there is a constant $c'$ such that 
\[
\frac{1}{\delta} \ell \left(\{s: g_sq' \text{ has a rigidity time in }[e^{t-\delta/2},e^{t}]\cap A\}\right)> c' \frac{\ell([e^{t-\delta/2},e^{t}]\cap A)}{e^{t-\delta/2}(e^{\delta/2}-1)} > \frac{\rho c'}{2}.
\]
In other words, a definite proportion of the flow line segment from $q'$ has rigidity times in $A$. 
 
The above inequality along with the inequality from mixing allow us to apply Fubini's theorem to conclude that 
\[
\mathbf{m}(\{q'\in B(q,r): q' \text{ has a rigidity time in }[e^{t-\delta/2},e^{t}]\cap A \})>
 \frac{1}{4} \rho c'  \mathbf{m}(B')\mathbf{m}\left(\cup_{s \in (0, \delta)} g_s V \right) > \frac{1}{36}\rho c' \mathbf{m}(B')  \mathbf{m}(B(q,r)).
 \]
However, the ball $B'$ depends on $M$ which in turn depends on $r$ and so $\mathbf{m}(B')$ may go to zero as $r$ goes to zero. Thus, the lower bound above for the proportion of $B(q,r)$ with rigidity times in $A$ is not uniform in $r$ and hence the proof of the corollary is not yet complete. 

Notice that just as in Lemma \ref{lem:key} the proof of the above inequality does not require that the periods be restricted between $1$ and $1+ \delta/4$. It works as long as the ratio of the largest to the smallest period is bounded above by $1+ \delta/4$ i.e., as long as the variation in the periods in $B$ is appropriately controlled. So to complete the proof of the corollary we follow the following steps: given a compact part $K$ and $\epsilon$ (but  before knowing $B(q,r)$) fix some $M> 1$ to obtain the ball $B$ as in Lemma \ref{lem:key}. Then given $B(q,r) \subset K$ obtain $\delta$ by Lemma \ref{foliate} and  partition $B$ into finitely many measurable sets $B_k$ where the period does not change by more than a factor of $1+ \delta/4$. Then for each $k$ we get
\[
\mathbf{m}(\{q'\in B(q,r): q' \text{ has a rigidity time in }[e^{t-\delta/2},e^{t}]\cap A \})>
 \frac{1}{36} \rho c'  \mathbf{m}(B_k)\mathbf{m}(B(q,r)).
\]
Summing over $k$ we get
\[
\mathbf{m}(\{q'\in B(q,r): q' \text{ has a rigidity time in }[e^{t-\delta/2},e^{t}]\cap A \})>
 \frac{1}{36} \rho c'  \mathbf{m}(B)\mathbf{m}(B(q,r)).
\]
proving the corollary.
\end{proof}

\begin{proof}[Proposition \ref{lots rig}] 
It suffices to show that for every $\epsilon>0$ the set of quadratic differentials that does not have 
an $\epsilon$ rigidity time in $A$ has no Lebesgue density points. By Corollary \ref{cor:local} for every $q'$ and $\epsilon>0$ there exists $c'$ such that for a small enough ball $B(q',r)$ we have
\[
\mathbf{m}(\{q'' \in B(q',r): q'' \text{ does not have an }\epsilon \text{ rigidity time in }A\})<\left(1- \frac{1}{36} \rho c'\right) \mathbf{m}(B(q',r)).
\]
This implies that $q'$ can not be a Lebesgue density point for the set of quadratic  differentials that do not have 
an $\epsilon$ rigidity time in $A$.

\end{proof}

\end{document}